\documentclass[12pt]{amsart}

\usepackage[all]{xy}    
\usepackage{graphicx,fullpage}
\usepackage{amssymb}
\usepackage{mathrsfs}
\usepackage{multirow}
\usepackage{float}
\usepackage{tikz-cd}
\usepackage{adjustbox}
\usepackage{amsthm}
\usepackage{accents}
\usepackage{mathtools}
\usepackage{mathptmx,enumitem,cite,array}
\usepackage[new]{old-arrows}
\usepackage{tikz}
\usetikzlibrary{matrix,arrows}
\usepackage{amsmath}
\usepackage[numbers,sort&compress]{natbib} 
\usepackage[bookmarksnumbered, bookmarksopen,
colorlinks,citecolor=blue,linkcolor=blue,backref]{hyperref}
\usepackage[utf8]{inputenc}
\usepackage{chngcntr}
\usepackage{apptools}
\usepackage[toc,page]{appendix}

\AtAppendix{\counterwithin{theorem}{section}}

\newtheorem{theorem}{Theorem}[section]

\theoremstyle{example}
\newtheorem{remark}[theorem]{Remark}

\numberwithin{equation}{section}

\newcommand{\beq}{\begin{equation}}
\newcommand{\eeq}{\end{equation}}

\newcommand{\TT}{\mathbb{T}}
\newcommand{\ZZ}{\mathbb{Z}}

\newcommand{\CC}{\mathbb{C}}

\DeclareMathOperator{\Tr}{Tr}

\newcommand{\C}{\mathbb{C}}

\newcommand{\Z}{\mathbb{Z}}

\newcommand{\T}{\mathbb{T}}

\newcommand{\bH}{\mathbb{H}}

\newcommand{\gH}{\mathfrak{H}}
\newcommand{\gI}{\mathfrak{I}}

\newcommand{\gch}{\mathrm{GCh}_H}
\newcommand{\hh}{\mathcal{H}}
\newcommand{\acal}{\mathcal{A}}

\DeclareMathOperator{\ch}{ch}

\newcommand {\be}{\begin{equation}}
\newcommand {\ee}{\end{equation}}
\newcommand{\h}{\begin{eqnarray*}}
\newcommand{\e}{\end{eqnarray*}}

\begin{document}


\title
{T-Duality, Jacobi Forms and Witten Gerbe Modules}

\author{Fei Han}
\address{Department of Mathematics,
National University of Singapore, Singapore 119076}
\email{mathanf@nus.edu.sg}

 \author{Varghese Mathai}
\address{School of Mathematical Sciences,
University of Adelaide, Adelaide 5005, Australia}
\email{mathai.varghese@adelaide.edu.au}

\date{}

\maketitle

\begin{abstract}
In this paper, we extend the T-duality Hori maps in \cite{BEM04a}, inducing isomorphisms of twisted cohomologies on T-dual circle bundles, to {\em graded Hori maps} and show that they induce isomorphisms of two-variable series of twisted cohomologies on the T-dual circle bundles, preserving Jacobi form properties. The composition of the graded Hori map with its dual is equal to  the Euler operator. We also construct  {\em Witten gerbe modules} arising from gerbe modules and show that their {\em graded twisted Chern characters} are Jacobi forms under an anomaly vanishing condition on gerbe modules, thereby giving interesting examples. 
\end{abstract}

\tableofcontents

\section*{Introduction}


Motivated by string theory, people have been attempting to generalize many concepts such as, vector bundles, Dirac operators, the
Atiyah-Singer index theory and so on to free loop spaces. Let $V$ be a rank $r$ complex vector bundle on a smooth manifold $M$ and $\widetilde V=V-\CC^r$ in the $K$-group of $M$.  In the theory of elliptic genera (\cite{HBJ, Hop, LS88, Liu95cmp, Liu95, O87, W86, W87, W4}), one considers the {\bf Witten bundles} $\mathcal{V}$ and $\mathcal{V}'$, elements in $K(M)[[q^{1/2}]]$, as follows (in the classical theory of elliptic genera, there are actually four such Witten bundles. For simplicity, we only discuss two of them here in the introduction. See (\ref{operations}) for explanation of the notions):
\be \label{Witten bundle}\mathcal{V}:=\bigotimes_{j=1}^{\infty}\Lambda_{-q^{j-{1/2}}}(\widetilde V)\otimes \bigotimes_{j=1}^{\infty}\Lambda_{-q^{j-{1/2}}}(\widetilde {\bar V}),\
\ \
\mathcal{V}':=\bigotimes_{j=1}^{\infty}\Lambda_{q^{j-{1/2}}}(\widetilde V)\otimes \bigotimes_{j=1}^{\infty}\Lambda_{q^{j-{1/2}}}(\widetilde {\bar V}).\ee
They are formally viewed as vector bundles over the free loop space $LM$. As Witten remarked in his lecture notes \cite{W4}, physically, this analogue is important because they arise in heterotic
string theory. 

Let $\{2\pi \sqrt{-1} x_i\},\, 1\leq i\leq r,$ be the formal Chern roots of $V$ and $q=e^{2\pi  \sqrt{-1}\tau}$, $\tau \in \bH$, the upper half plane. In terms of Jacobi theta functions (see Appendix), the Chern characters of $\mathcal{V}$ and $\mathcal{V}'$ can be expressed as
\be \ch (\mathcal{V})=\prod_{i=1}^r \frac{\theta_2(x_i, \tau)}{\theta_2(0, \tau)}\in H^{even}(M)[[q^{1/2}]], \ \  \ch (\mathcal{V'})=\prod_{i=1}^r \frac{\theta_3(x_i, \tau)}{\theta_3(0, \tau)}\in H^{even}(M)[[q^{1/2}]].\ee
Equip $V$ with a connection $\nabla^V$, using the Chern-Weil theory, the Chern characters $\ch \mathcal{V}, \ch \mathcal{V'}$ can be represented by holomorphic functions on $\bH$, taking values in the even degree closed differential forms on $M$. 

Suppose one has 
\be \frac{1}{2}p_1(V)=c_1(V)^2-2c_2(V)=\ch^{[4]}(V)=0,\ee where $p_1(V), c_1(V), c_2(V)$ and $\ch(V)$ stand for the first Pontryajin class, the first and second Chern class and the Chern character respectively, then the degree $p$ (with $p$ even) components
$\ch^{[p]} (\mathcal{V})$ and  $\ch^{[p]} (\mathcal{V'})$ are modular forms of weight $\frac{p}{2}$ over $\Gamma_0(2)$ and $\Gamma_\theta(2)$ respectively (see appendix for the meaning of notations and c.f. \cite{CH} for details).

Let $(E, E')$ be a gerbe module pair over manifold $M$ with flux $H$ (see Section \ref{Witten module} for the brief introduction). In this paper, we want to understand  the {\bf Witten gerbe modules} arising from $(E, E')$, an infinite dimensional analogue of the Witten bundles $\mathcal{V}, \mathcal{V'}$ constructed from a vector bundle $V$ as in (\ref{Witten bundle}). Such  Witten gerbe modules can be formally viewed as ``gerbe modules over loop space". We give the constructions of the Witten gerbe modules in Section \ref{Witten module}. 

When expanding $\mathcal{V}$ and $\mathcal{V'}$, one gets $q$-series with coefficients being virtual vector bundles manufactured out from $V$. If one expands the Witten gerbe modules, the coefficients in the $q$-series are virtual gerbe modules manufactured out from the gerbe module pair $(E, E')$. Since the construction involves exterior powers of $E$ and $E'$, these virtual gerbe modules are not only with twist $H$, but a mixed sum of modules with twists $mH, \, m\in \Z$ of various levels. Hence to take the Chern character of the Witten gerbe modules, we have to take into account of the levels of twists and apply the {\em twisted Chern character} (\cite{BCMMS}) of twist $mH$ when the module has twist $mH$ and mix them. This motivates us to introduce the {\bf graded twisted Chern character}, which is first constructed in \cite{HM19}. Unlike $\mathcal{V}$ and $\mathcal{V'}$ arising from vector bundle of finite rank, the gerbe modules $E, E'$ are infinite dimensional, and therefore there are some analytic difficulties to overcome for the convergence of the graded twisted Chern characters. We use the {\bf Holomorphic functional calculus} and {\bf Fredholm determinant} to deal with this. See Section \ref{Witten module} for details. 

The target space of the graded twisted Chern character are $q$-series with coefficients being differential forms on $M$, who are sums of $(d+mH)$-closed  differential forms for various $m$. To distinguish these $(d+mH)$-closed forms for various level $m$, we introduce a formal variable $y$ such that if a form $\omega$ is $(d+mH)$-closed, we write $\omega \cdot y^m$. Then the graded  twisted Character of the Witten gerbe modules take values in $\Omega^*(M)[[y, y^{-1}, q]]$ or $\Omega^*(M)[[y, y^{-1}, q^{1/2}]]. $ One can formally view them as spaces of twisted differential forms on free double loop space, and a model for the configuration space for Ramond-Ramond fields on this space. See some other loop space perspectives of T-duality in \cite{HM15, HM18, LM15}. 

We find that, analogous to that $\ch(\mathcal{V})$ and  $\ch (\mathcal{V'})$ have modularity under the anomaly vanishing condition $\ch^{[4]}(V)=0$, the graded twisted Chern character of the Witten gerbe modules are {\bf Jacobi forms} (see definitions in Section \ref{Jacobi}) over certain subgroups in $SL(2, \Z)$ under the anomaly vanishing conditions $Ch_H^{[2]}(E,E')=0$, $Ch_H^{[4]}(E, E')=0$. 

T-duality is an equivalence between two a priori distinct Type II string compactifications and backgrounds which nevertheless are indistinguishable from a physical point of view. These backgrounds can have different geometries, different fluxes, and strikingly can even be topologically distinct manifolds as first shown in \cite{BEM04a,BEM04b}. Let $Z, \widehat Z$ be circle bundles over $X$ and $H, \widehat H$ be H-fluxes on $Z, \widehat Z$ respectively such that $(Z,H)$ and 
$(\widehat Z, \widehat H)$ are T-dual pairs, see Section \ref{Tduality-review}. Then a key result proved is that the Hori map, 
$T: \Omega^{\overline{k}}(Z)^{\T} \to \Omega^{\overline{k+1}}(\hat{Z})^{\hat{\T}}$, where $k\ge 0$ and $\bar k = k\mod 2$, 
is a chain map which is an isometry upon choosing a Riemannian metric on $M$ and connections on the circle bundles $Z, \widehat Z$. This induces an isomorphism of twisted cohomology groups, $T: H^{\overline{k}}(Z, H) \to H^{\overline{k+1}}(\widehat Z, \widehat H)$. 

In view of the the above discussion about the target space of the graded twisted Chern characters and the Hori map in the T-duality, we generalise a central result  in \cite{BEM04a} as follows, referring to Section \ref{Tduality-review} and Section \ref{Jacobi} for more details. Let $\acal^{\bar k}(Z)^{\T}_{(d+mH)-cl}$ denote the space of holomorphic functions on $\bH$ except for a set of isolated points, which take values in $\Omega^{\bar k}(Z)^\T_{(d+mH)-cl}$, the $\T$-invariant $(d+mH)$-closed complex-valued differential forms on $Z$ with degree of parity $\bar k$.  Let $\hh^{\bar k}(Z, mH)$ denote the space of holomorphic functions on $\bH$ except for a set of isolated points, which take values in the twisted cohomology $H^{\bar k}(Z, mH)$ with degree of parity $\bar k$. In Section \ref{Jacobi}, we introduce the {\bf graded Hori map}, which is a chain map and prove that it is an isomorphism,
\be LT_*: \bigoplus_{m\in \Z, m\neq 0}\acal^{\bar k}(Z)^{\T}_{(d+mH)-cl}\cdot y^m\to \bigoplus_{m\in \Z, m\neq 0}\acal^{\overline{k+1}}(\hat Z)^{\hat \T}_{(d+m\hat H)-cl}\cdot y^m. \ee
Passing to cohomology, we get the induced isomorphism
\be LT: \bigoplus_{m\in \Z, m\neq 0} \hh^{\bar k}(Z, mH)\cdot y^m \to \bigoplus_{m\in \Z, m\neq 0} \hh^{\overline{k+1}}(\hat Z, m\hat H)\cdot y^m. \ee
One can similarly define the graded Hori map $\widehat{LT}_*$ and $ \widehat{LT}$ on the dual side.
We prove that if the graded Hori maps are applied twice,  
\be \widehat{LT}\circ LT, \ \ \ \  LT\circ  \widehat{LT}\ee
are equal to the Euler operator $-y\frac{\partial}{\partial y}$ in the variable $y$. It is easy to see that when restricted to the level $m=1$ component, the Euler operator $-y\frac{\partial}{\partial y}=-\mathrm{Id}$. So restricting ourself to level $m=1$ and a single point in $\bH$, we recover the results in \cite{BEM04a}, see Remark \ref{repack}. We will also prove that the 
graded Hori map sends Jacobi form elements to Jacobi form elements. We summarize these results in Theorem \ref{main}. 

The paper is organized as follows. In Section \ref{Tduality-review}, we give a brief review of T-duality. In Section \ref{Jacobi}, we introduce the graded Hori map and study the $T$-duality about it. In Section \ref{Witten module}, we construct the Witten gerbe modules arising from a gerbe module pair $(E, E')$ and study the graded twisted Chern character of them as well as the effect of the graded Hori map on them. We also study the odd analogue of Witten gerbe modules. Some basics about Jacobi theta functions used in the paper are provided in the Appendix.   

The results of this paper are for non-interacting strings. We plan to study the interacting case \cite{Wi86} in the near future, and are grateful to Chris Hull for pointing this out to V.M..

\bigskip

\subsection{Acknowledgements} Fei Han was partially supported by the grant AcRF R-146-000-218-112 from National University of Singapore. Varghese Mathai was 
supported by funding from the Australian Research Council, through the Australian Laureate Fellowship FL170100020. 
V.M. gave talks on this paper at the conferences, {\em Geometry of Quantum Fields and Strings}, University of Auckland,  January 10-12, 2020, 
and {\em M-theory and Mathematics}, NYU Abu Dhabi Institute, January 27-30, 2020. He thanks the conference participants for feedback.

\section{Review of T-duality}\label{Tduality-review}
Let $Z$ be a smooth manifold endowed with an $H$-flux which is a presentative in the degree 3 Deligne cohomology of $Z$, that is $H\in \Omega^3(Z)$ with integral periods (for simplicity, we drop the factor of $\frac{1}{2\pi \sqrt{-1}}$. 

Here we briefly review topological T-duality arising for the case of principal circle bundles with a H-flux. For such a case, one begins with a principal circle bundle $\pi: Z \to X$ whose first Chern class is given by $[F] \in H^2(X, \mathbb{Z})$, along with a H-flux given by some $[H] \in H^3(Z, \mathbb{Z})$. The aim is to then determine the corresponding data arising after an application of the T-duality transformation. 

One method for determining the T-dual data is by focusing on the Gysin sequence associated to the bundle $\pi : Z \to X$, given by:
\begin{center}
\begin{tikzcd}
\cdots \arrow[r] & H^{3}(X, \mathbb{Z}) \arrow[r, "\pi^{*}"] & H^3(Z, \mathbb{Z}) \arrow[r, "\pi_{*}"] & H^{2}(X, \mathbb{Z}) \arrow[r, "\lbrack F \rbrack \wedge"] & H^{4}(X, \mathbb{Z})  \arrow[r] & \cdots
\end{tikzcd}
\end{center}
Letting $[H] \in H^3(Z, \mathbb{Z})$, define $[\hat{F}] = \pi_*([H]) \in H^2(X, \mathbb{Z})$, and make the choice of some principal circle bundle $\hat{\pi} : \hat{Z} \to X$ with first Chern class $[\hat{F}]$. Having made such a choice, we then consider the Gysin sequence associated to the bundle $ \hat{Z}$ over $ X$,
\begin{center}
\begin{tikzcd}
\cdots \arrow[r] & H^{3}(X, \mathbb{Z}) \arrow[r, "\hat{\pi}^{*}"] & H^3(\hat{Z}, \mathbb{Z}) \arrow[r, "\hat{\pi}_{*}"] & H^{2}(X, \mathbb{Z}) \arrow[r, "\lbrack \hat{F}\rbrack \wedge"] & H^{4}(X, \mathbb{Z})  \arrow[r] & \cdots
\end{tikzcd}
\end{center}
Now using exactness, and the fact that $[F] \wedge [\hat{F}]= [F] \wedge \pi_*([H])=0$, there exists an $[\hat{H}] \in H^3(\hat{Z})$ such that $[F] = \hat{\pi}_*([\hat{H}])$. The following theorem gives a a global, geometric version of the Buscher rules \cite{Buscher}.

\begin{theorem}[\cite{BEM04a}]\label{TD1}
Let $\pi : Z \to X $ denote a principal circle bundle whose first Chern class is given by $[F] \in H^2(X, \mathbb{Z})$, and let $[H] \in H^3(Z)$ denote a H-flux on $Z$. 

Then there exists a T-dual bundle $\hat{\pi} : \hat{Z} \to X$ whose first Chern class is denoted $[\hat{F}] \in H^2(X, \mathbb{Z})$ and a T-dual H-flux on this bundle given by $[\hat{H}] \in H^3(\hat{Z}, \mathbb{Z})$, satisfying
\begin{align*}
[\hat{F}] &= \pi_*([H]),\\
[F] &= \hat{\pi}_*([\hat{H}]).
\end{align*}

Furthermore, letting $Z \times_X \hat{Z} = \{(a, b) \in Z \times \hat{Z} | \pi(a) = \hat{\pi}(b)\}$ and considering the following commutative diagram of bundle maps 
\begin{center}
\begin{tikzcd}
 &Z \times_X \hat{Z} \arrow[dl, "p"'] \arrow[rd, "\hat{p}"] & \\
Z\arrow[rd, "\pi"] & & \hat{Z}, \arrow[ld, "\hat{\pi}"'] \\
 & X & 
\end{tikzcd}
\end{center}
then if the two H-fluxes, $[H]$ and $[\hat{H}]$, satisfy
\begin{align*}
p^*([H]) = \hat{p}^*([\hat{H}]),
\end{align*}
the T-dual pair is unique up to bundle automorphism, and thus defines the T-duality transformation.
\end{theorem}

\begin{theorem}[\cite{BEM04a}]\label{TD2}
Let $A$, $\hat{A}$ denote connection forms on $Z$ and $\hat{Z}$ respectively, choose an invariant representative $H \in [H]$ and $\hat{H} \in [\hat{H}]$, and let $\big(\Omega^{*}(Z)^{S^1}, d + H\big)$ denote the $H$-twisted, $\mathbb{Z}_2$-graded differential complex of invariant differential forms.

Then the following Hori map:
\begin{align*}
T: (\Omega^{*}(Z)^{S^1}, d+H) &\to (\Omega^{*+1}(\hat{Z})^{\hat{S}^1}, -(d+{\hat{H})})\\
\omega \quad &\mapsto \int_{S^1}\omega \wedge e^{- \hat{A} \wedge A},
\end{align*}
is a chain map isomorphism between the twisted, $\mathbb{Z}_2$-graded complexes. Furthermore, it induces an isomorphism on the twisted cohomology:
\begin{align*}
T: H^{*}_{d+H}(Z) &\to H^{*+1}_{d+\hat{H}}(\hat{Z}).
\end{align*}
\end{theorem}

\section{Jacobi forms and graded Hori maps}\label{Jacobi}
Let $\Gamma$ be a subgroup of $SL(2, \Z)$ of finite index. Let $L$ be an integral lattice in $\C$ preserved by $\Gamma$. Denote $\bH$ the upper half plane. A {\bf (meromorphic) Jacobi form} (c.f. \cite{EZ, Liu95}) of weight $s$ and index $l$ over $L\rtimes \Gamma$ is a (meromorphic) function $J(z, \tau)$ on $\C \times \bH$ such that \newline
(i) $J\left(\frac{z}{c\tau+d}, \frac{a\tau+b}{c\tau+d}\right)=(c\tau+d)^se^{2\pi \sqrt{-1} l(cz^2/(c\tau+d))}J(z,\tau)$;\newline
(ii) $J(z+\lambda \tau+\mu, \tau)=e^{-2\pi \sqrt{-1}l(\lambda^2\tau+2\lambda z))}J(a, \tau),$
where 
$$(\lambda, \mu)\in L, \ \left(\begin{array}{cc} a&b\\ c&d\end{array}\right)\in \Gamma. $$

In this paper, we will use a slight extension of the above definition of Jacobi forms, namely, (i) we will allow $J(z, \tau)$ to take values in the differential forms on a manifold $M$; (ii) as $J(z, \tau)$ takes values in differential forms, we don't require the singular points be poles but only remain undefined. 

Let $M$ be a manifold with $H$-flux. Let $\acal^{\bar k}(M)_{(d+mH)-cl}$ denote the space of holomorphic functions on $\bH$ except for a set of isolated points, which take values in $\Omega^{\bar k}(M)_{(d+mH)-cl}$, the $(d+mH)$-closed forms on $M$ with degree parity $\bar k$.  Let $\hh^{\bar k}(M, mH)$ denote the space of holomorphic functions on $\bH$ except for a set of isolated points, which take values in $H^{\bar k}(M, mH)$.

Denote $q=e^{2\pi \sqrt{-1} \tau}, \tau\in \mathbb{H}$ and $y=e^{-2\pi \sqrt{-1} z}, \, z\in \C$. On $M$, consider the 2-variable series  
$$\omega(z, \tau)\in \bigoplus_{m\in \ZZ} \hh^{\bar k}(M, mH)\cdot y^m$$ with the following properties: $\omega(z, \tau)$ is represented by 
\be   \sum_{m\in \ZZ}\omega_m(\tau)y^m,\ee
with $\omega_m(\tau)\in \acal^{\bar k}(M)_{(d+mH)-cl}, m\in \ZZ$ such that the degree $p$ (with $\bar p=\bar k)$ component 
\be  \sum_{m\in \ZZ}\omega_m(\tau)^{[p]}y^m\ee
is the expansion at $y=0$ of a Jacobi form of weight $\frac{p+\bar k}{2}$ and index 0 over $L\rtimes \Gamma$. Denote the abelian group of all such $\omega(z, \tau)$  by $\mathcal{J}_0^{\bar k}(M, H; L, \Gamma)$.

Now consider the situation of T-duality with pair $(Z, H), (\hat Z, \hat H)$ as in Section \ref{Tduality-review}. For $m\in \Z$, define the {\bf level $m$ Hori map} by
\be T_{*, m}(G)= \int_\T e^{-mA\wedge \hat A } G,\ee
for $G$ is an $\T$-invariant form on $Z$ and $(d+mH)G=0$. As we have 
\be m\hat H=mH+d(mA\wedge \hat A),\ee it is not hard to see that $T_{*, m}G$ is a $\hat \T$-invariant form on $\hat Z$ and $$(d+m\hat H)(T_{*, m}(G))=0,$$ similar to the $m=1$ case.

Denote $\acal^{\bar k}(Z)^{\T}_{(d+mH)-cl}$ the space of holomorphic functions on $\bH$ except for a set of isolated points, which take values in $\Omega^{\bar k}(Z)^\T_{(d+mH)-cl}$, the $\T$-invariant $(d+mH)$-closed forms on $Z$ with degree parity $\bar k$. Denote $\acal^{\bar k}(\hat Z)^{\hat \T}_{(d+m\hat H)-cl}$ the similar stuff on the dual side. 
Define the {\bf graded Hori map} 
\be LT_*: \bigoplus_{m\in \Z}\acal^{\bar k}(Z)^{\T}_{(d+mH)-cl}\cdot y^m\to \bigoplus_{m\in \Z}\acal^{\overline{k+1}}(\hat Z)^{\hat \T}_{(d+m\hat H)-cl}\cdot y^m \ee
by
\be  LT_*\left(\sum_{m\in \Z}\omega_m(\tau)y^m\right)=\sum_{m\in \Z} T_{*, m}(\omega_m(\tau))y^m,\ee
for 
$$\sum_{m\in \Z}\omega_m(\tau)y^m \in \bigoplus_{m\in \Z}\acal^{\bar k}(Z)^{\T}_{(d+mH)-cl}\cdot y^m.$$ 
Passing to cohomology, we have the graded Hori map
\be LT: \bigoplus_{m\in \Z} \hh^{\bar k}(Z, mH)\cdot y^m \to \bigoplus_{m\in \Z} \hh^{\overline{k+1}}(\hat Z, m\hat H)\cdot y^m. \ee

One can  similarly define on the dual side,
\be \widehat{LT}_*: \bigoplus_{m\in \Z}\acal^{\bar k}(\hat Z)^{\hat \T}_{(d+m\hat H)-cl}\cdot y^m\to \bigoplus_{m\in \Z}\acal^{\overline{k+1}}(Z)^{\T}_{(d+mH)-cl}\cdot y^m \ee
and 
\be \widehat{LT}: \bigoplus_{m\in \Z} \hh^{\bar k}(\hat Z, mH)\cdot y^m \to \bigoplus_{m\in \Z} \hh^{\overline{k+1}}(Z, mH)\cdot y^m. \ee

\begin{remark} $Z$ and $\hat Z$ are circle bundles. As treated in \cite{BEM04a}, one considers $\T$-invariant forms on $Z$ and $\hat T$-invariant forms on $\hat Z$. However, on the level of cohomology, invariant twisted cohomology is same as the usual twisted cohomology.  

\end{remark}

\begin{theorem}\label{main} Let $H(\bH)$ denote the space of holomorphic functions on $\bH$.  The following identities hold:\newline
(i) \be \label{compo} \widehat{LT}\circ LT=-y\frac{\partial}{\partial y}=-\frac{\sqrt{-1}}{2\pi }\frac{\partial}{\partial z}, \ \ LT\circ \widehat{LT}=-y\frac{\partial}{\partial y}=-\frac{\sqrt{-1}}{2\pi }\frac{\partial}{\partial z};\ee
in particular when restricting to $m\neq 0$ parts, we get isomorphisms of $H(\bH)$ modules,
\be LT_*: \bigoplus_{m\in \Z, m\neq 0}\acal^{\bar k}(Z)^{\T}_{(d+mH)-cl}\cdot y^m\to \bigoplus_{m\in \Z, m\neq 0}\acal^{\overline{k+1}}(\hat Z)^{\hat \T}_{(d+m\hat H)-cl}\cdot y^m; \ee
\be \widehat{LT}_*: \bigoplus_{m\in \Z, m\neq 0}\acal^{\bar k}(\hat Z)^{\hat \T}_{(d+m\hat H)-cl}\cdot y^m\to \bigoplus_{m\in \Z, m\neq 0}\acal^{\overline{k+1}}(Z)^{\T}_{(d+mH)-cl}\cdot y^m.  \ee
(ii) Restrcting to the Jacobi forms, we have 
\be LT\left(\mathcal{J}_0^{\bar k}(Z, H; L, \Gamma))\right)\subseteq \mathcal{J}_0^{\overline {k+1}}(\hat Z, \hat H; L, \Gamma) \ee and therefore 
get a map of abelian groups,
\be LT: \mathcal{J}_0^{\bar k}(Z, H; L, \Gamma))\to \mathcal{J}_0^{\overline {k+1}}(\hat Z, \hat H; L, \Gamma);\ee 
dually, we have
\be \widehat{LT}\left(\mathcal{J}^{\bar k}_0(\hat Z, \hat H; L, \Gamma)\right)\subseteq\mathcal{J}_0^{\overline{k+1}}(Z, H; L, \Gamma)\ee
and therefore 
get a map of abelian groups,
\be \widehat{LT}\colon \mathcal{J}^{\bar k}_0(\hat Z, \hat H; L, \Gamma)\to\mathcal{J}_0^{\overline{k+1}}(Z, H; L, \Gamma).\ee

\end{theorem}
\begin{proof} (i) Take a representative 
\be   \sum_{m\in \Z}\omega_m(\tau)y^m,\ee
with $\omega_m(\tau)\in \acal^{\bar k}(Z)_{(d+mH)-cl}^\TT, m \in \ZZ$. Then each $\omega_m(\tau)$ must be of the form
\be  F_m(\tau)+G_m(\tau)A,\ee
with $F_m(\tau), G_m(\tau)$ being holomorphic functions on $\bH$ valued in $\Omega^{\bar k}(X)$ and $\Omega^{\overline{k+1}}(X)$ with isolated singular points respectively; and $A$ is the connection on $Z$ as in Theorem \ref{TD2}.

Applying the level $m$ Hori map, we get
\be 
\begin{split}
&T_{*, m}(F_m(\tau)+G_m(\tau)A)\\
=& \int_\T e^{-mA\wedge \hat A } (F_m(\tau)+G_m(\tau)A)\\
=& (-1)^{k+1}(G_m(\tau)+mF_m(\tau)\hat A).
\end{split}
\ee

Applying the reverse level $m$ Hori map, we get
\be 
\begin{split}
&\hat T_{*, m}( (-1)^{k+1}(G_m(\tau)+mF_m(\tau)\hat A))\\
=& (-1)^{k+1} \int_{\hat \T} e^{mA\wedge \hat A } (G_m(\tau)+mF_m(\tau)\hat A))\\
=& (-1)^{k+1}((-1)^kmG_m(\tau)+(-1)^kmF_m(\tau))\\
=&-m(F_m(\tau)+G_m(\tau)A).
\end{split}
\ee

Therefore we see that 
\be 
\begin{split}
&\widehat{LT}_*\circ LT_*\left(\sum_{m\in \Z}\omega_m(\tau)y^m\right)\\
=& \sum_{m\in \Z} \hat T_{*, m}\circ T_{*, m}(\omega_m(\tau))y^m\\
=&-\sum_{m\in \Z} m\omega_m(\tau)y^m\\
=&-y\frac{\partial}{\partial y}\left(\sum_{m\in \Z}\omega_m(\tau)y^m\right)\\
=&-\frac{\sqrt{-1}}{2\pi }\frac{\partial}{\partial z} \left(\sum_{m\in \Z}\omega_m(\tau)y^m\right).
\end{split}
\ee
So the first equality in (\ref{compo}) is proved. Similarly, one can prove the second equality.

$\, $

(ii) Let $\omega(z, \tau)\in \mathcal{J}_0^{\bar k}(Z, H; L, \Gamma))$ be represented by 
$$\sum_{m\in \Z}\omega_m(\tau)y^m,$$
with $\omega_m(\tau)\in \acal^{\bar k}(Z)_{(d+mH)-cl}^\TT, m\in \Z$ such that the degree $p$ component 
$$\sum_{m\in \Z}\omega_m(\tau)^{[p]}y^m$$
is a Jacobi form of weight $\frac{p+\bar k}{2}$ and index 0 over $L\rtimes \Gamma$. 

Let $\omega_m(\tau)=F_m(\tau)+G_m(\tau)A.$ Then it is not hard to see that 
\be\sum_{m\in \Z} F_m(\tau)^{[p]}y^m \ee
is a Jacobi form of weight $\frac{p+\bar k}{2}$ and index 0 over $L\rtimes \Gamma$ and 
\be \sum_{m\in \Z} G_m(\tau)^{[p-1]}y^m \ee 
is a Jacobi form of weight $\frac{p+\bar k}{2}$ and index 0 over $L\rtimes \Gamma$. 

By the proof of (i), we see that 
\be LT_*\left(\sum_{m\in \Z}\omega_m(\tau)y^m\right)=(-1)^{k+1}\left( \sum_{m\in \Z} G_m(\tau)y^m +\frac{1}{2\pi \sqrt{-1}} \frac{\partial }{\partial z} \left(\sum_{m\in \Z} F_m(\tau)\hat A\cdot y^m \right)\right). \ee

Clearly  $\sum_{m\in \Z} G_m(\tau)^{[p]}y^m$ is a Jacobi form of weight $\frac{p+\overline{k+1}}{2}$ and index 0 over $L\rtimes \Gamma$.

From (i), (ii) in the definition of Jacobi forms, it is not hard to see that if $J(z, \tau)$ is a Jacobi form of weight $s$ and index 0 over $L\rtimes \Gamma$, then $\frac{\partial }{\partial z} J(z, \tau)$ is still a  Jacobi form of weight $s$ and index 0 over $L\rtimes \Gamma$.  Therefore 
$$\frac{\partial }{\partial z} \left(\sum_{m\in \Z} (F_m(\tau)\hat A)^{[p]}\cdot y^m \right)=\frac{\partial }{\partial z} \left(\sum_{m\in \Z} (F_m(\tau))^{[p-1]}\cdot y^m \right)\cdot \hat A$$ 
is a Jacobi form of weight $\frac{p+\overline{k+1}}{2}$ and index 0 over $L\rtimes \Gamma$. 

This shows that 
$$(-1)^{k+1}\left( \sum_{m\in \Z} G_m(\tau)y^m +\frac{1}{2\pi \sqrt{-1}} \frac{\partial }{\partial z} \left(\sum_{m\in \Z} F_m(\tau)\hat A\cdot y^m \right)\right)\in \mathcal{J}_0^{\overline {k+1}}(\hat Z, \hat H; L, \Gamma).$$
Hence 
$$LT\left(\mathcal{J}_0^{\bar k}(Z, H; L, \Gamma))\right)\subseteq \mathcal{J}_0^{\overline {k+1}}(\hat Z, \hat H; L, \Gamma).$$

One can similarly prove the dual side. 

\end{proof}

\begin{remark}\label{repack}
It follows from the results of this section that T-duality according to \cite{BEM04a} can be repackaged as follows. Let $\sum_{m\in \ZZ}\omega_m y^m$, 
with $\omega_m \in \Omega^{\bar k}(Z)_{(d+mH)-cl}, m\in \ZZ$, and let $T_{*, m}$
be the level $m$ Hori map, and $T_y$ be the sum $\oplus_m T_{*, m}$. Similarly, let  $\widehat T_{*, m}$
be the level $m$ Hori map, and $\widehat T_y$ be the sum $\oplus_m \widehat T_{*, m}$. Then it follows formally
from Theorem \ref{main} that $ \widehat{T_y}\circ T_y=-y\frac{\partial}{\partial y}$ and $T_y\circ \widehat{T_y}=-y\frac{\partial}{\partial y}$,
where the Euler operator appears and this agrees with \cite{BEM04a} when $m=1$. This repackaging is particularly useful when
one considers the Chern character of tensor products of gerbe modules.
\end{remark}

\section{Witten gerbe modules} \label{Witten module}
In this section, we construct the {\bf Witten gerbe modules} arising from a gerbe module pair $(E, E')$ inipired by the classical theory of elliptic genera. The graded twisted Chern character of them give examples of elements in  $\mathcal{J}_0^{\bar k}(M, H; L, \Gamma)$, under an anomaly cancellation condition. We will also study effect of the graded Hori map on these Jacobi forms.

\subsection{Even case} 
Let $M$ be an oriented closed smooth manifold of dimension $2r$. Let $H$ be a closed 3-form on $M$ with integral periods. Let $B_\alpha\in \Omega^2(U_\alpha)$ such that $dB_\alpha=H|_{\alpha}$. Let $A_{\alpha\beta}\in \Omega(U_{\alpha\beta})$ such that $B_\alpha-B_\beta=dA_{\alpha\beta}.$ Let $\{(L_{\alpha\beta}, d+A_{\alpha\beta})\}$ be geometric realization of the gerbe (with connection). Then we have
\beq\label{gerbeconn}
(\nabla^L_{\alpha\beta})^2 = F^L_{\alpha\beta} = B_\beta- B_\alpha.
\eeq

Let $E=\{E_{\alpha}\}$
be a collection of (infinite dimensional) separable Hilbert bundles $E_{\alpha}\to U_{\alpha}$ whose structure group is reduced to
$U_{\gI}$, which are unitary operators on the model Hilbert space $\gH$ of the form identity + trace class operator.
Here $\gI$ denotes the Lie algebra of  $U_{\gI}$, the trace class operators on $\gH$.
In addition, assume that on the overlaps $U_{\alpha\beta}$ 
there are isomorphisms
\beq
\phi_{\alpha\beta}: L_{\alpha\beta} \otimes E_\beta \cong E_\alpha,
\eeq
which are consistently defined on
triple overlaps because of the gerbe property. Then $\{E_{\alpha}\}$ is said to be a {\em gerbe module} for the gerbe
$\{L_{\alpha\beta}\}$. A {\em gerbe module connection} $\nabla^E$ is a collection of connections $\{\nabla^E_{\alpha}\}$ of the form $\nabla^E_{\alpha} = d + A_\alpha^E$, where $A_\alpha^E
\in \Omega^1(U_\alpha)\otimes \gI$ whose curvature $F^{E_\alpha}$ on the overlaps $U_{\alpha\beta}$ satisfies
\beq
\phi_{\alpha\beta}^{-1}(F^{E_\alpha}) \phi_{\alpha\beta} =  F^{L_{\alpha\beta}} I  +    F^{E_\beta}.
\eeq
Using equation \eqref{gerbeconn}, this becomes
\beq \label{patch}
\phi_{\alpha\beta}^{-1}( B_\alpha I + F^E_\alpha ) \phi_{\alpha\beta} = B_{\beta} I  +    F^E_\beta.
\eeq
It follows that $\exp(-B)\Tr\left(\exp(-F^E) - I\right)$ is a globally well defined differential form on $M$
of even degree. Notice that $\Tr(I)=\infty$ and that is why we need to consider the subtraction.

Let $E=\{E_{\alpha}\}$ and $E'=\{E'_{\alpha}\}$ 
be {gerbe modules} for the gerbe $\{L_{\alpha\beta}\}$. Then an element of twisted K-theory $K^0(M, H)$
is represented by the pair $(E, E')$, see \cite{BCMMS}. Two such pairs $(E, E')$ and $(G, G')$ are equivalent
if $E\oplus G' \oplus K \cong E' \oplus G \oplus K$ as gerbe modules for some gerbe module $K$ for the gerbe $\{L_{\alpha\beta}\}$.
We can assume without loss of generality that these gerbe modules $E, E'$ are modeled on the same Hilbert space
$\gH$, after a choice of isomorphism if necessary.

Suppose that $\nabla^E, \nabla^{E'}$ are gerbe module connections on the gerbe modules $E, E'$ respectively. Then one can define the {\em twisted Chern character} as
\be \label{twistedChern}
\begin{split}
&Ch_H: K^0(M, \mathcal{G}) \to H^{even}(M, H)\\
&Ch_H(E, E')= \exp(-B)\Tr\left(\exp(-F^E) - \exp(-F^{E'})\right)
\end{split}
\ee
That this is a well defined homomorphism is explained in \cite{BCMMS}. Clearly the degree 0 component of the $Ch_H(E, E')$ is 0. The degree 2 component is 
\be Ch_H^{[2]}(E, E')=\Tr[F^{E}-F^{E'}]=\{\Tr[F^{E_\alpha}-F^{E_\alpha'}]\}. \ee
The degree 4 component is 
\be Ch_H^{[4]}(E, E')=\frac{\Tr[(B+F)^2-(B+F')^2]}{2}=\left\{ \frac{\Tr[(B_\alpha+F^{E_\alpha})^2-(B_\alpha+F^{E'_\alpha})^2]}{2}\right\}. \ee

Recall that for an indeterminate $t$ (c.f. \cite{A67}), 
\be \label{operations} \Lambda_t(E)=\CC
|_M+tE+t^2\wedge^2(E)+\cdots,\ \ \ S_t(E)=\CC |_M+tE+t^2
S^2(E)+\cdots, \ee are the total exterior and
symmetric powers of $E$ respectively. The following relations between these two operations
hold, \be S_t(E)=\frac{1}{\Lambda_{-t}(E)},\ \ \ \
 \Lambda_t(E-F)=\frac{\Lambda_t(E)}{\Lambda_t(F)}.\ee

On $U_\alpha$, define
\be \Theta(E_\alpha)=\bigotimes_{u=1}^\infty
\Lambda_{-q^{u}}(E_{\alpha})\otimes\bigotimes_{u=1}^\infty
\Lambda_{-q^{u}}(\bar E_{\alpha}) .\ee
In the Fourier expansion of $\Theta(E_{\alpha})$,  the coefficient of $q^{n}$ is integral linear combination of terms of the form $$\wedge^{i_1}(E_\alpha)\otimes\wedge^{i_2}(E_\alpha)\otimes \cdots \wedge^{i_k}(E_\alpha)\otimes \wedge^{j_1}(\bar E_\alpha)\otimes \wedge^{j_2}(\bar E_\alpha)\otimes \cdots \wedge^{j_l}(\bar E_\alpha). $$ Pick out the the terms such that $(i_1+i_2+\cdots i_k)-(j_1+j_2+\cdots j_l)=m$ and denote their sum by $W_{m, n}(E)$. Note that for each $m$, there are only finite many nonvanishing $W_{m, n}(E_{\alpha})$. Then we can express
\be  \label{arrange0} \Theta(E_{\alpha})=\sum_{m\in \Z} (\sum_{n=0}^\infty W_{m, n}(E_{\alpha})q^{n}).   \ee 
The isomorphism $\phi_{\alpha\beta}$ induces an isomorphism (which we still denote by $\phi_{\alpha\beta}$ to abuse notation)
$$ \phi_{\alpha\beta}: L_{\alpha\beta}^{\otimes m} \otimes \wedge^{i_1}(E_\alpha)\otimes\wedge^{i_2}(E_\alpha)\otimes \cdots \wedge^{i_k}(E_\alpha)\otimes \wedge^{j_1}(\bar E_\alpha)\otimes \wedge^{j_2}(\bar E_\alpha)\otimes \cdots \wedge^{j_l}(\bar E_\alpha)$$
$$\to \wedge^{i_1}(E_\beta)\otimes\wedge^{i_2}(E_\beta)\otimes \cdots \wedge^{i_k}(E_\beta)\otimes \wedge^{j_1}(\bar E_\beta)\otimes \wedge^{j_2}(\bar E_\beta)\otimes \cdots \wedge^{j_l}(\bar E_\beta).$$
Therefore we can see that $\{W_{m,n}(E_{\alpha})\}$ gives a gerbe module for the gerbe $(mH, mB_\alpha, mA_{\alpha\beta})$ for each $m\in \Z$. Denote them by $W_{m,n}(E)$.

On $U_\alpha$, define
\be \Theta_1(E_{\alpha})=\bigotimes_{u=1}^\infty
\Lambda_{q^u}(E_{\alpha})\otimes \bigotimes_{u=1}^\infty
\Lambda_{q^u}(\bar E_{\alpha}),\ee
\be \Theta_2(E_\alpha)=\bigotimes_{v=1}^\infty
\Lambda_{-q^{v-{1\over2}}}(E_{\alpha})\otimes \bigotimes_{v=1}^\infty
\Lambda_{-q^{v-{1\over2}}}(\bar E_{\alpha}),\ee
\be \Theta_3(E_\alpha)=\bigotimes_{v=1}^\infty
\Lambda_{q^{v-{1\over2}}}(E_{\alpha})\otimes \bigotimes_{v=1}^\infty
\Lambda_{q^{v-{1\over2}}}(\bar E_{\alpha}).\ee
One similarly can express
\be  \Theta_1(E_{\alpha})=\sum_{m\in \Z} (\sum_{n=0}^\infty A_{m, n}(E_{\alpha})q^{n}),\ee
\be  \Theta_2(E_{\alpha})=\sum_{m\in \Z} (\sum_{n=0}^\infty B_{m, n}(E_{\alpha})q^{n/2}), \ee 
 \be  \Theta_3(E_{\alpha})=\sum_{m\in \Z} (\sum_{n=0}^\infty C_{m, n}(E_{\alpha})q^{n/2}),\ee 
where $\{A_{m,n}(E_{\alpha})\}$, $\{B_{m,n}(E_{\alpha})\}$ and $\{C_{m,n}(E_{\alpha})\}$ are gerbe modules for the gerbe $(mH, mB_\alpha, mA_{\alpha\beta})$ for each $m\in \Z$. Denote them by $A_{m,n}(E), B_{m,n}(E)$ and $C_{m,n}(E)$ respectively.

We call the systems 
\be 
\begin{split}
&\Theta(E)=\{\Theta(E_\alpha)\},\ \ \Theta_1(E)=\{\Theta_1(E_\alpha)\},\\
& \Theta_2(E)=\{\Theta_2(E_\alpha)\},\ \ \Theta_3(E)=\{\Theta_3(E_\alpha)\}
\end{split}
\ee  
{\bf Witten gerbe modules} arising from the gerbe module $E$. 

Now consider the quotient 
\be \label{wtheta} \frac{\Theta(E_\alpha)}{\Theta(E'_\alpha)}=\frac{\bigotimes_{u=1}^\infty
\Lambda_{-q^{u}}(E_{\alpha})\otimes\bigotimes_{u=1}^\infty
\Lambda_{-q^{u}}(\bar E_{\alpha}) }{\bigotimes_{u=1}^\infty
\Lambda_{-q^{u}}(E'_{\alpha})\otimes \bigotimes_{u=1}^\infty
\Lambda_{-q^{u}}(\bar E'_{\alpha})}.\ee
As the bottom starts from the trivial bundle $\CC$, we can use the power series expansion of $\frac{1}{1+x}$ to formally expand at $x=0$ for the bottom and arrange as in (\ref{arrange0}) in the following way
\be \frac{\Theta(E_\alpha)}{\Theta(E'_\alpha)}=\frac{\bigotimes_{u=1}^\infty
\Lambda_{-q^{u}}(E_{\alpha})\otimes\bigotimes_{u=1}^\infty
\Lambda_{-q^{u}}(\bar E_{\alpha}) }{\bigotimes_{u=1}^\infty
\Lambda_{-q^{u}}(E'_{\alpha})\otimes \bigotimes_{u=1}^\infty
\Lambda_{-q^{u}}(\bar E'_{\alpha})}=\sum_{m\in \Z} (\sum_{n=0}^\infty W_{m, n}(E_{\alpha}, E'_{\alpha})q^{n}),\ee
where $W_{m, n}(E_{\alpha}, E'_{\alpha})=W_{m, n}(E_{\alpha}, E'_{\alpha})^+\ominus W_{m, n}(E_{\alpha}, E'_{\alpha})^-$ is a virtual bundle. $W_{m, n}(E_{\alpha}, E'_{\alpha})^+$ is a Hilbert bundle of finite sum of tensored exterior powers of $E_\alpha$ and exterior powers of $E'_\alpha$ with each summand having total tensor power $m$. $W_{m, n}(E_{\alpha}, E'_{\alpha})^-$ has similar property. As when $E_\alpha=E'_\alpha$, 
$\frac{\Theta(E_\alpha)}{\Theta(E'_\alpha)}=\CC$, one sees that neither $W_{m, n}(E_{\alpha}, E'_{\alpha})^+$ or $W_{m, n}(E_{\alpha}, E'_{\alpha})^-$ vanishes. 

The isomorphism $\phi_{\alpha\beta}$ deduces an isomorphism (which we still denote by $\phi_{\alpha\beta}$ to abuse notation)
$$ \phi_{\alpha\beta}: L_{\alpha\beta}^{\otimes m} \otimes W_{m, n}(E_{\alpha}, E'_{\alpha})^\pm\to W_{m, n}(E_{\alpha}, E'_{\alpha})^\pm. $$

Let $F^{W_{m, n}(E_{\alpha}, E'_{\alpha})^\pm}$ be the induced curvature on the bundle $W_{m, n}(E_{\alpha}, E'_{\alpha})^\pm$. As $F^{E_\alpha}$ is trace class and $W_{m, n}(E_{\alpha}, E'_{\alpha})^\pm$ is constructed from tensor powers of $E_\alpha$, we see that $F^{W_{m, n}(E_{\alpha}, E'_{\alpha})^\pm}$ is also trace class. 

Then similar in (\ref{twistedChern}), one can see that 
\be \exp{(-mB_\alpha)} \Tr\left(\exp(-W_{m, n}(E_{\alpha}, E'_{\alpha})^+) - \exp(-W_{m, n}(E_{\alpha}, E'_{\alpha})^-)\right) \ee
patch together to be a $d+mH$ closed form even form on $M$. We simply denote this twisted Chern character by
\be \mathrm{Ch}_{mH}(W_{m, n}(E, E')). \ee

Denote the collection $\left\{\frac{\Theta(E_\alpha)}{\Theta(E'_\alpha)}\right\}$ by $\frac{\Theta(E)}{\Theta(E')}$.

Similarly construct
\be \label{wtheta1} \frac{\Theta_1(E_\alpha)}{\Theta_1(E'_\alpha)}=\frac{\bigotimes_{u=1}^\infty
\Lambda_{q^{u}}(E_{\alpha})\otimes\bigotimes_{u=1}^\infty
\Lambda_{q^{u}}(\bar E_{\alpha}) }{\bigotimes_{u=1}^\infty
\Lambda_{q^{u}}(E'_{\alpha})\otimes \bigotimes_{u=1}^\infty
\Lambda_{q^{u}}(\bar E'_{\alpha})}=\sum_{m\in \Z} (\sum_{n=0}^\infty A_{m, n}(E_{\alpha}, E'_{\alpha})q^{n}); \ee

\be\label{wtheta2} \frac{\Theta_2(E_\alpha)}{\Theta_2(E'_\alpha)}=\frac{\bigotimes_{v=1}^\infty
\Lambda_{-q^{v-{1\over2}}}(E_{\alpha})\otimes\bigotimes_{v=1}^\infty
\Lambda_{-q^{v-{1\over2}}}(\bar E_{\alpha}) }{\bigotimes_{v=1}^\infty
\Lambda_{-q^{v-{1\over2}}}(E'_{\alpha})\otimes \bigotimes_{v=1}^\infty
\Lambda_{-q^{v-{1\over2}}}(\bar E'_{\alpha})}=\sum_{m\in \Z} (\sum_{n=0}^\infty B_{m, n}(E_{\alpha}, E'_{\alpha})q^{n/2}); \ee

\be \label{wtheta3}\frac{\Theta_3(E_\alpha)}{\Theta_3(E'_\alpha)}=\frac{\bigotimes_{v=1}^\infty
\Lambda_{q^{v-{1\over2}}}(E_{\alpha})\otimes\bigotimes_{v=1}^\infty
\Lambda_{q^{v-{1\over2}}}(\bar E_{\alpha}) }{\bigotimes_{v=1}^\infty
\Lambda_{q^{v-{1\over2}}}(E'_{\alpha})\otimes \bigotimes_{v=1}^\infty
\Lambda_{q^{v-{1\over2}}}(\bar E'_{\alpha})}=\sum_{m\in \Z} (\sum_{n=0}^\infty C_{m, n}(E_{\alpha}, E'_{\alpha})q^{n/2})\ee
and denote the collection $\left\{\frac{\Theta_i(E_\alpha)}{\Theta_i(E'_\alpha)}\right\}$ by $\frac{\Theta_i(E)}{\Theta_i(E')}, i=1,2 ,3$. One can similarly define the twisted Chern characters for $A_{m, n}(E, E'), B_{m, n}(E, E')$ and $C_{m, n}(E, E')$. 

Define the {\bf graded twisted Chern character} (c.f. \cite{HM19}): 
\be 
\gch\left(\frac{\Theta(E)}{\Theta(E')}\right)=\sum_{m\in \Z}(\sum_{n=0}^\infty\mathrm{Ch}_{mH}(W_{m, n}(E, E'))q^n)y^m\in \bigoplus_{m\in \Z}\Omega^{ev}(M)_{(d+mH)-cl}[q]\cdot y^m,\ee
\be 
\gch\left(\frac{\Theta_1(E)}{\Theta_1(E')}\right)=\sum_{m\in \Z}(\sum_{n=0}^\infty\mathrm{Ch}_{mH}(A_{m, n}(E, E'))q^n)y^m\in \bigoplus_{m\in \Z}\Omega^{ev}(M)_{(d+mH)-cl}[q]\cdot y^m,
\ee
\be \gch\left(\frac{\Theta_2(E)}{\Theta_2(E')}\right)=\sum_{m\in \Z}(\sum_{n=0}^\infty\mathrm{Ch}_{mH}(B_{m, n}(E, E'))q^{n/2})y^m\in \bigoplus_{m\in \Z}\Omega^{ev}(M)_{(d+mH)-cl}[q^{1/2}]\cdot y^m,\ee
\be \gch\left(\frac{\Theta_3(E)}{\Theta_3(E')}\right)=\sum_{m\in \Z}(\sum_{n=0}^\infty\mathrm{Ch}_{mH}(C_{m, n}(E, E'))q^{n/2})y^m\in \bigoplus_{m\in \Z}\Omega^{ev}(M)_{(d+mH)-cl}[q^{1/2}]\cdot y^m.\ee

It is not hard to see that when $\sinh(\pi \sqrt{-1} z)\neq 0$ (and therefore $\sinh(\pi \sqrt{-1}z+\pi \sqrt{-1}B_\alpha)\neq 0$ since $B_\alpha$ is a differential form), 
$$f(x)=\frac{\sinh(\pi \sqrt{-1}z+\pi \sqrt{-1}B_\alpha+\frac{x}{2})}{\sinh(\pi \sqrt{-1}z+\pi \sqrt{-1}B_\alpha)}-1$$ is a holomorphic function for $x\in \CC$ and $f(0)=0$. 
As $F^{E_\alpha}$ is trace class, and the Banach algebra of trace class operators is closed under the {\bf holomorphic functional calculus}, we see that 
\be \frac{\sinh(\pi \sqrt{-1}z+\pi \sqrt{-1}B_\alpha+F^{E_\alpha}/2)}{\sinh(\pi \sqrt{-1}z+\pi \sqrt{-1}B_\alpha)}-1 \ee
is also trace class. Let 
\be \det \left(I+ \left[\frac{\sinh(\pi \sqrt{-1}z+\pi \sqrt{-1}B_\alpha+F^{E_\alpha}/2)}{\sinh(\pi \sqrt{-1}z+\pi \sqrt{-1}B_\alpha)}-1 \right] \right)  \ee
be the {\bf Fredholm determinant} \cite{BS}. 

The form
\be \det\left(\frac{\sinh(\pi \sqrt{-1}z+\pi \sqrt{-1}B_\alpha+F^{E_\alpha}/2)}{\sinh(\pi \sqrt{-1}z+\pi \sqrt{-1}B_\alpha+F^{E'_\alpha}/2)}\right),\ee
then can be understood as
\be \frac{\det \left(I+ \left[\frac{\sinh(\pi \sqrt{-1}z+\pi \sqrt{-1}B_\alpha+F^{E_\alpha}/2)}{\sinh(\pi \sqrt{-1}z+\pi \sqrt{-1}B_\alpha)}-1 \right] \right)}{\det \left(I+ \left[\frac{\sinh(\pi \sqrt{-1}z+\pi \sqrt{-1}B_\alpha+F^{E'_\alpha}/2)}{\sinh(\pi \sqrt{-1}z+\pi \sqrt{-1}B_\alpha)}-1 \right] \right) }. \ee

By $(\ref{patch})$, It is not hard to see that $\left\{\det\left(\frac{\sinh(\pi \sqrt{-1}z+\pi \sqrt{-1}B_\alpha+F^{E_\alpha}/2)}{\sinh(\pi \sqrt{-1}z+\pi \sqrt{-1}B_\alpha+F^{E'_\alpha}/2)}\right)\right\}$ patch to be a global form in 
$\bigoplus_{m\in \Z}\Omega^{ev}(M)_{(d+mH)-cl}\cdot y^m $. 
Denote it by 
\be \label{sinh} \det\left(\frac{\sinh(\pi \sqrt{-1}z+\pi \sqrt{-1}B+F^E/2)}{\sinh(\pi \sqrt{-1}z+\pi \sqrt{-1}B+F^{E'}/2)}\right). \ee

One can similarly construct
\be \det\left(\frac{\cosh(\pi \sqrt{-1}z+\pi \sqrt{-1}B+F^E/2)}{\cosh(\pi \sqrt{-1}z+\pi \sqrt{-1}B+F^{E'}/2)}\right)\in \bigoplus_{m\in \Z}\Omega^{ev}(M)_{(d+mH)-cl}\cdot y^m. \ee

Define
\be 
\begin{split} 
&W(E, E')=\det\left(\frac{\sinh(\pi \sqrt{-1}z+\pi \sqrt{-1}B+F^{E}/2)}{\sinh(\pi \sqrt{-1}z+\pi \sqrt{-1}B+F^{E'}/2)}\right)\gch\left(\frac{\Theta(E)}{\Theta(E')}\right)\\
&\in \bigoplus_{m\in \Z}\Omega^{ev}(M)_{(d+mH)-cl}[q]\cdot y^m, \\
&A(E, E')=\det\left(\frac{\cosh(\pi \sqrt{-1}z+\pi \sqrt{-1}B+F^{E}/2)}{\cosh(\pi \sqrt{-1}z+\pi \sqrt{-1}B+F^{E'}/2)}\right)\gch\left(\frac{\Theta_1(E)}{\Theta_1(E')}\right)\\
&\in \bigoplus_{m\in \Z}\Omega^{ev}(M)_{(d+mH)-cl}[q]\cdot y^m, \\
&B(E, E')=\gch\left(\frac{\Theta_2(E)}{\Theta_2(E')}\right)\in \bigoplus_{m\in \Z}\Omega^{ev}(M)_{(d+mH)-cl}[q^{1/2}]\cdot y^m\\
&C(E, E')=\gch\left(\frac{\Theta_3(E)}{\Theta_3(E')}\right)\in \bigoplus_{m\in \Z}\Omega^{ev}(M)_{(d+mH)-cl}[q^{1/2}]\cdot y^m.
\end{split}
\ee

Passing cohomology, denote the corresponding cohomology classes by
\be 
\begin{split} 
&\mathcal{W}(E, E')\in\bigoplus_{m\in \Z}H^{ev}(M, d+mH)[q]\cdot y^m , \\
&\mathcal{A}(E, E')\in \bigoplus_{m\in \Z}H^{ev}(M, d+mH)[q]\cdot y^m, \\
&\mathcal{B}(E, E')\in \bigoplus_{m\in \Z}H^{ev}(M, d+mH)[q^{1/2}]\cdot y^m\\
&\mathcal{C}(E, E')\in  \bigoplus_{m\in \Z}H^{ev}(M, d+mH)[q^{1/2}]\cdot y^m.
\end{split}
\ee

\begin{theorem} If $Ch_H^{[2]}(E,E')=0$ and $Ch_H^{[4]}(E, E')=0$, then
\be \mathcal{W}(E, E')\in \mathcal{J}_0^{\bar 0}(M, H; \ZZ^2, SL(2, \Z)),\ee
\be \mathcal{A}(E, E')\in \mathcal{J}_0^{\bar 0}(M, H; \ZZ^2, \Gamma_0(2)),\ee
\be \mathcal{B}(E, E')\in \mathcal{J}_0^{\bar 0}(M, H; \ZZ^2, \Gamma^0(2)),\ee
\be \mathcal{C}(E, E')\in \mathcal{J}_0^{\bar 0}(M, H; \ZZ^2, \Gamma_\theta(2)).\ee
\end{theorem}
\begin{proof} 
When $\theta(z, \tau)\neq 0$ (and therefore $\theta(z+B_\alpha, \tau)\neq 0$), consider the function 
\be g(x)=\frac{\theta(z+B_\alpha+x, \tau)}{\theta(z+B_\alpha, \tau)}-1.\ee
It is a holomorphic function for $x\in \CC$ and $g(0)=0$. As $F^{E_\alpha}$ is trace class, by the holomorphic functional calculus, we see that 
$\frac{\theta(z+B_\alpha+F^{E_\alpha}, \tau)}{\theta(z+B_\alpha, \tau)}-1$
is also trace class. 
Let 
\be \det\left(I+\left[ \frac{\theta(z+B_\alpha+F^{E_\alpha}, \tau)}{\theta(z+B_\alpha, \tau)}-1\right] \right) \ee
be the Fredholm determinant. 

As before the form
\be \det\left(\frac{\theta(z+B_\alpha+F^{E_\alpha}, \tau)}{\theta(z+B_\alpha+F^{E'_\alpha}, \tau)}\right)\ee
can be understood as
\be \frac{\det \left(I+ \left[\frac{\theta(z+B_\alpha+F^{E_\alpha}, \tau)}{\theta(z+B_\alpha, \tau)}-1 \right] \right)}{\det \left(I+ \left[\frac{\theta(z+B_\alpha+F^{E'_\alpha}, \tau)}{\theta(z+B_\alpha, \tau)}-1 \right] \right) }. \ee

Similar to (\ref{sinh}), the forms $\left\{\det\left(\frac{\theta(z+B_\alpha+F^{E_\alpha}, \tau)}{\theta(z+B_\alpha+F^{E'_\alpha}, \tau)}\right)\right\}$ patch together to be a global form in $\bigoplus_{m\in \Z}\Omega^{ev}(M)_{(d+mH)-cl}[q]\cdot y^m$.

Like the finite dimensional case for projective elliptic genera \cite{HM19}, the following identity holds, 
\be \label{W}W(E, E')=\det\left(\frac{\theta(z+B+F^{E}, \tau)}{\theta(z+B+F^{E'}, \tau)}\right).\ee

Applying the transformation laws (\ref{theta-tran1}, \ref{theta-tran2}), when $Ch_H^{[2]}(E,E')=0$ and $Ch_H^{[4]}(E, E')=0$, we have for the even degree $p$ component,
\be 
\begin{split}
&\det\left(\frac{\theta(\frac{z}{\tau}+B+F^{E}, -\frac{1}{\tau})}{\theta(\frac{z}{\tau}+B+F^{E'}, -\frac{1}{\tau})}\right)^{[p]}\\
=&\left\{e^{\pi\sqrt{-1}\tau \Tr\left[\left(\frac{z}{\tau}+B+F^{E}\right)^2 -\left(\frac{z}{\tau}+B+F^{E'}\right)^2\right]}\det\left(\frac{{1\over\sqrt{-1}}\left({\tau\over
\sqrt{-1}}\right)^{1/2} \theta(z+\tau(B+F^{E}), \tau)}{{1\over\sqrt{-1}}\left({\tau\over
\sqrt{-1}}\right)^{1/2} \theta(z+\tau(B+F^{E'}), \tau)}\right)\right\}^{[p]}\\
=&\det\left(\frac{\theta(z+\tau(B+F^{E}), \tau)}{\theta(z+\tau(B+F^{E'}), \tau)}\right)^{[p]}\\
=&\tau^{p/2}\det\left(\frac{\theta(z+B+F^{E}, \tau)}{\theta(z+B+F^{E'}, \tau)}\right);
\end{split}
\ee

\be \det\left(\frac{\theta(z+B+F^{E}, \tau+1)}{\theta(z+B+F^{E'}, \tau)+1}\right)^{[p]}=\det\left(\frac{e^{\pi \sqrt{-1}\over 4}\theta(z+B+F^{E}, \tau)}{e^{\pi \sqrt{-1}\over 4}\theta(z+B+F^{E'}, \tau)}\right)^{[p]}=\det\left(\frac{\theta(z+B+F^{E}, \tau)}{\theta(z+B+F^{E'}, \tau)}\right)^{[p]};\ee

\be \det\left(\frac{\theta(z+1+B+F^{E}, \tau)}{\theta(z+1+B+F^{E'}, \tau)}\right)^{[p]}=\det\left(\frac{-\theta(z+B+F^{E}, \tau)}{-\theta(z+B+F^{E'}, \tau)}\right)^{[p]}=\det\left(\frac{\theta(z+B+F^{E}, \tau)}{\theta(z+B+F^{E'}, \tau)}\right)^{[p]};\ee

\be 
\begin{split}
&\det\left(\frac{\theta(z+\tau+B+F^{E}, \tau)}{\theta(z+\tau+B+F^{E'}, \tau)}\right)^{[p]}\\
=&\left\{e^{-2\pi \sqrt{-1}[\Tr(z+B+F^E)-\Tr(z+B+F^{E'}) ]}\det\left(\frac{e^{-\pi \sqrt{-1}\tau}\theta(z+\tau+B+F^{E}, \tau)}{e^{-\pi \sqrt{-1}\tau}\theta(z+\tau+B+F^{E'}, \tau)}\right)\right\}^{[p]}\\
=&\det\left(\frac{\theta(z+B+F^{E}, \tau)}{\theta(z+B+F^{E'}, \tau)}\right)^{[p]}.
\end{split}
\ee

As $SL(2, \Z)$ is generated by $S:\tau\rightarrow-\frac{1}{\tau},\, T:\tau\rightarrow\tau+1$, it is not hard to see from the definition of Jacobi forms that 
$$ \mathcal{W}(E, E')\in \mathcal{J}_0^{\bar 0}(M, H; \ZZ^2, SL(2, \Z)). $$

Similarly, we can show that 
\be 
\begin{split} \label{ABC}
&A(E, E')=\det\left(\frac{\theta_1(z+B+F^{E}, \tau)}{\theta_1(z+B+F^{E'}, \tau)}\right),  \\
&B(E, E')=\det\left(\frac{\theta_2(z+B+F^{E}, \tau)}{\theta_2(z+B+F^{E'}, \tau)}\right), \\
&C(E, E')=\det\left(\frac{\theta_3(z+B+F^{E}, \tau)}{\theta_3(z+B+F^{E'}, \tau)}\right).
\end{split}
\ee
Then by the transformations laws (in the appendix) of the corresponding theta functions, the fact that the generators of $\Gamma_0(2)$ are $T,ST^2ST$, the generators
of $\Gamma^0(2)$ are $STS,T^2STS$  and the generators of
$\Gamma_\theta$ are $S$, $T^2$ and the definition of Jacobi forms, we have 
$$ \mathcal{A}(E, E')\in \mathcal{J}_0^{\bar 0}(M, H; \ZZ^2, \Gamma_0(2)),$$
$$ \mathcal{B}(E, E')\in \mathcal{J}_0^{\bar 0}(M, H; \ZZ^2, \Gamma^0(2)),$$
$$ \mathcal{C}(E, E')\in \mathcal{J}_0^{\bar 0}(M, H; \ZZ^2, \Gamma_\theta(2)).$$

\end{proof}

In the T-dual situation, when $M=Z$ or $\hat Z$, applying Theorem \ref{main} and the proof of the above theorem, we have the following, 
\begin{theorem} 
\be  \label{deri-theta}
\begin{split} 
&(\widehat {LT}_*\circ LT_*)W(E, E')\\
=&-\frac{\sqrt{-1}}{2\pi }\frac{\partial}{\partial z}\det\left(\frac{\theta(z+B+F^{E}, \tau)}{\theta(z+B+F^{E'}, \tau)}\right)\\
=&-\frac{\sqrt{-1}}{2\pi }W(E, E')\Tr\left[ \frac{\theta'(z+B+F^{E}, \tau)}{\theta(z+B+F^{E}, \tau)}-\frac{\theta'(z+B+F^{E'}, \tau)}{\theta(z+B+F^{E'}, \tau)}\right];
\end{split}
\ee

\be \label{deri-theta1}
\begin{split} 
&(\widehat {LT}_*\circ LT_*)A(E, E')\\
=&-\frac{\sqrt{-1}}{2\pi }\frac{\partial}{\partial z}\det\left(\frac{\theta_1(z+B+F^{E}, \tau)}{\theta_1(z+B+F^{E'}, \tau)}\right)\\
=&-\frac{\sqrt{-1}}{2\pi }A(E, E')\Tr\left[ \frac{\theta_1'(z+B+F^{E}, \tau)}{\theta_1(z+B+F^{E}, \tau)}-\frac{\theta_1'(z+B+F^{E'}, \tau)}{\theta_1(z+B+F^{E'}, \tau)}\right];
\end{split}
\ee

\be \label{deri-theta2}
\begin{split} 
&(\widehat {LT}_*\circ LT_*)B(E, E')\\
=&-\frac{\sqrt{-1}}{2\pi }\frac{\partial}{\partial z}\det\left(\frac{\theta_2(z+B+F^{E}, \tau)}{\theta_2(z+B+F^{E'}, \tau)}\right)\\
=&-\frac{\sqrt{-1}}{2\pi }B(E, E')\Tr\left[ \frac{\theta_2'(z+B+F^{E}, \tau)}{\theta_2(z+B+F^{E}, \tau)}-\frac{\theta_2'(z+B+F^{E'}, \tau)}{\theta_2(z+B+F^{E'}, \tau)}\right];
\end{split}
\ee

\be \label{deri-theta3}
\begin{split} 
&(\widehat {LT}_*\circ LT_*)C(E, E')\\
=&-\frac{\sqrt{-1}}{2\pi }\frac{\partial}{\partial z}\det\left(\frac{\theta_3(z+B+F^{E}, \tau)}{\theta_3(z+B+F^{E'}, \tau)}\right)\\
=&-\frac{\sqrt{-1}}{2\pi }C(E, E')\Tr\left[ \frac{\theta_3'(z+B+F^{E}, \tau)}{\theta_3(z+B+F^{E}, \tau)}-\frac{\theta_3'(z+B+F^{E'}, \tau)}{\theta_3(z+B+F^{E'}, \tau)}\right].
\end{split}
\ee

\begin{remark} When $\theta(z, \tau)\neq 0$, the trace part in (\ref{deri-theta}) is well defined. Actually,  
$$\Tr\left[ \frac{\theta'(z+B+F^{E}, \tau)}{\theta(z+B+F^{E}, \tau)}-\frac{\theta'(z+B+F^{E'}, \tau)}{\theta(z+B+F^{E'}, \tau)}\right]$$ 
should be understood as 
$$\Tr\left[ \left(\frac{\theta'(z+B+F^{E}, \tau)}{\theta(z+B+F^{E}, \tau)}-\frac{\theta'(z+B, \tau)}{\theta(z+B, \tau)}\right)-\left(\frac{\theta'(z+B+F^{E'}, \tau)}{\theta(z+B+F^{E'}, \tau)}-\frac{\theta'(z+B, \tau)}{\theta(z+B, \tau)}\right)\right].$$ 
Note that if $\theta(z, \tau)\neq 0$, then $\theta(z+B, \tau)\neq 0$. As the manifold $M$ is finite dimensional, the Taylor expansion of $\frac{1}{\frac{\theta(z+B+F^{E}, \tau)}{\theta(z+B, \tau)}}$ as a polynomial of $F^E$ has only finite many terms and leading term 1.  On the other hand, The function 
$$\theta'(z+B+x, \tau)-\theta'(z+B, \tau)$$
a holomorphic function for $x\in \CC$ and takes 0 when $x=0$ and therefore, $$\theta'(z+B+F^E, \tau)-\theta'(z+B, \tau)$$ is traceable. Therefore 
$$\frac{\theta'(z+B+F^{E}, \tau)}{\theta(z+B+F^{E}, \tau)}-\frac{\theta'(z+B, \tau)}{\theta(z+B, \tau)}=\frac{\frac{\theta'(z+B+F^{E}, \tau)}{\theta(z+B, \tau)}}{\frac{\theta(z+B+F^{E}, \tau)}{\theta(z+B, \tau)}}-\frac{\theta'(z+B, \tau)}{\theta(z+B, \tau)}$$ is traceable. The term about $F^{E'}$ is similarly traceable. Formulas (\ref{deri-theta1}), (\ref{deri-theta2}) and (\ref{deri-theta3}) all have similar meaning. 

\end{remark}

\end{theorem}

\subsection{Odd case} In the following, we give examples of Jacobi forms of odd degrees. Let $\{(H, B_\alpha, A_{\alpha\beta})\}$ be a gerbe with connection on $M$ as in Section  3.1.  Let $E=\{E_\alpha\}$ be a $U_{tr}$ gerbe module with module connection $\nabla^E=\{\nabla^{E_\alpha}\}$.  Let $\phi=\{\phi_\alpha: E_\alpha\to E_\alpha\}$ be an automorphism of the gerbe module $E$ that respects the $U_{tr}$ gerbe module structure, that is, $\phi \in U_{tr}(E)$, then $\phi^{-1}\nabla^E\phi$ is another module connection for $E$. As explained in \cite{MS}, 
\be (\phi_\alpha^{-1}F^{E_\alpha}\phi_\alpha+B_\alpha )^k- (F^{E_\alpha}+B_\alpha )^k \ee
are differential forms with values in the trace class endomorphisms of $E_\alpha$ and 
\be \Tr[(\phi_\alpha^{-1}F^{E_\alpha}\phi_\alpha+B_\alpha )^k- (F^{E_\alpha}+B_\alpha )^k] \ee
patch together to be an even degree differential form on $M$. Denote it by $\Tr[(\phi^{-1}F^{E}\phi+B )^k- (F^{E}+B )^k].$

Let $\nabla^E(s)=s\phi^{-1}\nabla^E\phi+(1-s)\nabla^E$ be a path joining $\phi^{-1}\nabla^E\phi$ and $\nabla^E$. Let $A(s)=\partial_s\nabla^E(s)=\phi^{-1}\nabla^E\phi-\nabla^E$, which satisfies 
\be A_\alpha(s)=\psi^{-1}_{\alpha\beta}A_\beta(s) \psi_{\alpha\beta}. \ee

Following \cite{MS}, one defines the odd Chern character form
\be Ch_H(\nabla^E, \phi)=-\exp{(-B)}\int_0^1ds \Tr[A(\phi)\exp(-F^{E}(s))]. \ee
Clearly the degree 1 term is
\be Ch_H^{[1]}(\nabla^E, \phi)=-\Tr[A(\phi)], \ee
and the degree 3 terms is
\be Ch_H^{[3]}(\nabla^E, \phi)=\int_0^1 ds\Tr[A(\phi)(B+F^{E}(s))]. \ee

In view of $(\ref{W})$ and $(\ref{ABC})$ and Chern-Simons transgression,  define
\be
\begin{split}
&W(\nabla^E, \phi)=-\int_0^1ds \Tr\left[A(\phi)\frac{\theta'(z+B+F^{E}(s), \tau)}{\theta(z+B+F^{E}(s), \tau)}\right]\in \bigoplus_{m\in \Z}\Omega^{odd}(M)_{(d+mH)-cl}[q]\cdot y^m,\\
&A(\nabla^E, \phi)=-\int_0^1ds \Tr\left[A(\phi)\frac{\theta_1'(z+B+F^{E}(s), \tau)}{\theta_1(z+B+F^{E}(s), \tau)}\right]\in \bigoplus_{m\in \Z}\Omega^{odd}(M)_{(d+mH)-cl}[q]\cdot y^m,\\
&B(\nabla^E, \phi)=-\int_0^1ds \Tr\left[A(\phi)\frac{\theta_2'(z+B+F^{E}(s), \tau)}{\theta_2(z+B+F^{E}(s), \tau)}\right]\in \bigoplus_{m\in \Z}\Omega^{odd}(M)_{(d+mH)-cl}[q^{1/2}]\cdot y^m,\\
&C(\nabla^E, \phi)=-\int_0^1ds \Tr\left[A(\phi)\frac{\theta_3'(z+B+F^{E}(s), \tau)}{\theta_3(z+B+F^{E}(s), \tau)}\right]\in \bigoplus_{m\in \Z}\Omega^{odd}(M)_{(d+mH)-cl}[q^{1/2}]\cdot y^m.
\end{split}
\ee
 \begin{remark} Note that if $\theta(z, \tau)\neq 0$, then $\theta(z+B, \tau)\neq 0$. As the manifold $M$ is finite dimensional, the Taylor expansion of $\frac{1}{\frac{\theta(z+B+F^{E}(s), \tau)}{\theta(z+B, \tau)}}$ as a polynomial of $F^E(s)$ has only finite many terms and leading term 1. Also $A(\phi)$ is trace class. Hence $A(\phi)\frac{\theta'(z+B+F^{E}(s), \tau)}{\theta(z+B+F^{E}(s), \tau)}$ is also trace class. 
 
 \end{remark}
\begin{theorem} If $Ch_H^{[1]}(\nabla^E, \phi)=0$ and $Ch_H^{[3]}(\nabla^E, \phi)=0$, then
\be \mathcal{W}(\nabla^E, \phi)\in \mathcal{J}_0^{\bar 1}(M, H; \ZZ^2, SL(2, \Z)),\ee
\be \mathcal{A}(\nabla^E, \phi)\in \mathcal{J}_0^{\bar 1}(M, H; (\ZZ)^2, \Gamma_0(2)),\ee
\be \mathcal{B}(\nabla^E, \phi)\in \mathcal{J}_0^{\bar 1}(M, H; (\ZZ)^2, \Gamma^0(2)),\ee
\be \mathcal{C}(\nabla^E, \phi)\in \mathcal{J}_0^{\bar 1}(M, H; (\ZZ)^2, \Gamma_\theta(2)).\ee
\end{theorem}
\begin{proof} By the transformation laws (\ref{theta-tran1}), we have 
\be \begin{split} &\theta'(v,\tau+1)=e^{\pi
\sqrt{-1}\over 4}\theta'(v,\tau), \\
&\theta'\left(v,-{1}/{\tau}\right)={1\over\sqrt{-1}}\left({\tau\over
\sqrt{-1}}\right)^{1/2} e^{\pi\sqrt{-1}\tau v^2}(2\pi\sqrt{-1}\tau
v\theta\left(\tau v,\tau\right)+\tau\theta'(\tau v, \tau));\\
&\theta_1'(v,\tau+1)=e^{\pi \sqrt{-1}\over
4}\theta_1'(v,\tau),\\
&\theta_1'\left(v,-{1}/{\tau}\right)=\left({\tau\over
\sqrt{-1}}\right)^{1/2} e^{\pi\sqrt{-1}\tau v^2}(2\pi\sqrt{-1}\tau
v\theta_2\left(\tau v,\tau\right)+\tau\theta_2'(\tau v, \tau));\\
&\theta_2'(v,\tau+1)=\theta_3'(v,\tau),\\
&\theta_2'\left(v,-{1}/{\tau}\right)=\left({\tau\over
\sqrt{-1}}\right)^{1/2} e^{\pi\sqrt{-1}\tau v^2}(2\pi\sqrt{-1}\tau
v\theta_1\left(\tau v,\tau\right)+\tau\theta_1'(\tau v, \tau));\\
&\theta_3'(v,\tau+1)=\theta_2'(v,\tau),\\
&\theta_3'\left(v,-{1}/{\tau}\right)=\left({\tau\over
\sqrt{-1}}\right)^{1/2} e^{\pi\sqrt{-1}\tau v^2}(2\pi\sqrt{-1}\tau
v\theta_3\left(\tau v,\tau\right)+\tau\theta_3'(\tau v,
\tau)).\end{split}\ee

By the transformation laws (\ref{theta-tran2}), we have

\begin{equation} 
\begin{split}
&\theta'(z+1,\tau )=-\theta' (z,\tau ),\ \theta' (z+\tau ,\tau )=-%
e^{-\pi \sqrt{-1}(\tau+2z)}\theta' (z, \tau )+2\pi \sqrt{-1}e^{-\pi \sqrt{-1}(\tau+2z)}\theta(z, \tau ),\\
&\theta_{1}'(z+1,\tau )=-\theta _{1}'(z,\tau ),\ \theta _{1}'(z+\tau ,\tau )=%
e^{-\pi \sqrt{-1}(\tau+2z)}\theta _{1}'(z,\tau )-2\pi \sqrt{-1}e^{-\pi \sqrt{-1}(\tau+2z)}\theta _{1}(z,\tau ),\\
&\theta _{2}'(z+1,\tau )=\theta _{2}'(z,\tau ),\ \theta _{2}'(z+\tau ,\tau )=-%
e^{-\pi \sqrt{-1}(\tau+2z)}\theta _{2}'(z,\tau )+2\pi \sqrt{-1}e^{-\pi \sqrt{-1}(\tau+2z)}\theta_2(z, \tau ),\\
&\theta _{3}'(z+1,\tau )=\theta _{3}'(z,\tau ),\ \theta _{3}'(z+\tau ,\tau )=%
e^{-\pi \sqrt{-1}(\tau+2z)}\theta _{3}'(z,\tau )-2\pi \sqrt{-1}e^{-\pi \sqrt{-1}(\tau+2z)}\theta _{3}(z,\tau ).
\end{split}
\end{equation}

Applying these transformation laws and the condition $Ch_H^{[1]}(\nabla^E, \phi)=0$ and $Ch_H^{[3]}(\nabla^E, \phi)=0$, we have for the odd degree $p$ component, 
\be 
\begin{split}
&\left(\int_0^1ds \Tr\left[A(\phi)\frac{\theta'(\frac{z}{\tau}+B+F^{E}(s), -\frac{1}{\tau})}{\theta(\frac{z}{\tau}+B+F^{E}(s), -\frac{1}{\tau})}\right]\right)^{[p]}\\
=&\left(\int_0^1ds \Tr\left[A(\phi)\left(\frac{\tau\theta'(z+\tau(B+F^{E}(s)), \tau)}{\theta(z+\tau(B+F^{E}(s)), \tau)}+2\pi \sqrt{-1}\tau(z+B+F^E(s))\right)\right]\right)^{[p]}\\
=&\tau^{\frac{p+1}{2}}\left(\int_0^1ds \Tr\left[A(\phi)\frac{\theta'(\frac{z}{\tau}+B+F^{E}(s), -\frac{1}{\tau})}{\theta(\frac{z}{\tau}+B+F^{E}(s), -\frac{1}{\tau})}\right]\right)^{[p]};
\end{split}
\ee

\be 
\begin{split}
&\left(\int_0^1ds \Tr\left[A(\phi)\frac{\theta'(z+B+F^{E}(s), \tau+1)}{\theta(z+B+F^{E}(s), \tau+1)}\right]\right)^{[p]}\\
=&\left(\int_0^1ds \Tr\left[A(\phi)\frac{\theta'(z+B+F^{E}(s), {\tau})}{\theta(z+B+F^{E}(s), {\tau})}\right]\right)^{[p]};
\end{split}
\ee

\be 
\begin{split}
&\left(\int_0^1ds \Tr\left[A(\phi)\frac{\theta'(z+1+B+F^{E}(s), \tau)}{\theta(z+1+B+F^{E}(s), \tau)}\right]\right)^{[p]}\\
=&\left(\int_0^1ds \Tr\left[A(\phi)\frac{\theta'(z+\tau+B+F^{E}(s), \tau)}{\theta(z+\tau+B+F^{E}(s), \tau)}\right]\right)^{[p]};
\end{split}
\ee

\be 
\begin{split}
&\left(\int_0^1ds \Tr\left[A(\phi)\frac{\theta'(z+\tau+B+F^{E}(s), \tau)}{\theta(z+\tau+B+F^{E}(s), \tau)}\right]\right)^{[p]}\\
=&\left(\int_0^1ds \Tr\left[A(\phi)\left(\frac{\theta'(z+B+F^{E}(s), {\tau})}{\theta(z+B+F^{E}(s), {\tau})}+2\pi \sqrt{-1}\right)\right]\right)^{[p]}\\
=&\left(\int_0^1ds \Tr\left[A(\phi)\frac{\theta'(z+\tau+B+F^{E}(s), \tau)}{\theta(z+\tau+B+F^{E}(s), \tau)}\right]\right)^{[p]}.
\end{split}
\ee

As $SL(2, \Z)$ is generated by $S$ and $T$, it is not hard to see by the definition of Jacobi forms that 
$$ \mathcal{W}(E, \phi)\in \mathcal{J}_0^{\bar 1}(M, H; \ZZ^2, SL(2, \Z)). $$

By the transformations laws the corresponding theta functions, the fact that the generators of $\Gamma_0(2)$ are $T,ST^2ST$, the generators
of $\Gamma^0(2)$ are $STS,T^2STS$  and the generators of
$\Gamma_\theta$ are $S$, $T^2$ and the definition of Jacobi forms, we have 
$$ \mathcal{A}(E, \phi)\in \mathcal{J}_0^{\bar 1}(M, H; \ZZ^2, \Gamma_0(2)),$$
$$ \mathcal{B}(E, \phi)\in \mathcal{J}_0^{\bar 1}(M, H; \ZZ^2, \Gamma^0(2)),$$
$$ \mathcal{C}(E, \phi)\in \mathcal{J}_0^{\bar 1}(M, H; \ZZ^2, \Gamma_\theta(2)).$$

\end{proof}

\appendix
\section{The Jacobi theta functions}\label{Appendix}

A general reference for this appendix is \cite{Ch85}.

Let $$ SL_2(\mathbb{Z}):= \left\{\left.\left(\begin{array}{cc}
                                      a&b\\
                                      c&d
                                     \end{array}\right)\right|a,b,c,d\in\mathbb{Z},\ ad-bc=1
                                     \right\}
                                     $$
 as usual be the modular group. Let
$$S=\left(\begin{array}{cc}
      0&-1\\
      1&0
\end{array}\right), \ \ \  T=\left(\begin{array}{cc}
      1&1\\
      0&1
\end{array}\right)$$
be the two generators of $ SL_2(\mathbb{Z})$. Their actions on
$\mathbb{H}$ are given by
$$ S:\tau\rightarrow-\frac{1}{\tau}, \ \ \ T:\tau\rightarrow\tau+1.$$

Let
$$ \Gamma_0(2)=\left\{\left.\left(\begin{array}{cc}
a&b\\
c&d
\end{array}\right)\in SL_2(\mathbb{Z})\right|c\equiv0\ \ (\rm mod \ \ 2)\right\},$$

$$ \Gamma^0(2)=\left\{\left.\left(\begin{array}{cc}
a&b\\
c&d
\end{array}\right)\in SL_2(\mathbb{Z})\right|b\equiv0\ \ (\rm mod \ \ 2)\right\}$$

$$ \Gamma_\theta=\left\{\left.\left(\begin{array}{cc}
a&b\\
c&d
\end{array}\right)\in SL_2(\mathbb{Z})\right|\left(\begin{array}{cc}
a&b\\
c&d
\end{array}\right)\equiv\left(\begin{array}{cc}
1&0\\
0&1
\end{array}\right) \mathrm{or} \left(\begin{array}{cc}
0&1\\
1&0
\end{array}\right)\ \ (\rm mod \ \ 2)\right\}$$
be the three modular subgroups of $SL_2(\mathbb{Z})$. It is known
that the generators of $\Gamma_0(2)$ are $T,ST^2ST$, the generators
of $\Gamma^0(2)$ are $STS,T^2STS$  and the generators of
$\Gamma_\theta$ are $S$, $T^2$. (cf. \cite{Ch85}).


The four Jacobi theta-functions (c.f. \cite{Ch85}) defined by
infinite products are

\be \theta(v,\tau)=2q^{1/8}\sin(\pi v)\prod_{j=1}^\infty[(1-q^j)(1-e^{2\pi \sqrt{-1}v}q^j)(1-e^{-2\pi
\sqrt{-1}v}q^j)], \ee
\be \theta_1(v,\tau)=2q^{1/8}\cos(\pi v)\prod_{j=1}^\infty[(1-q^j)(1+e^{2\pi \sqrt{-1}v}q^j)(1+e^{-2\pi
\sqrt{-1}v}q^j)], \ee
 \be \theta_2(v,\tau)=\prod_{j=1}^\infty[(1-q^j)(1-e^{2\pi \sqrt{-1}v}q^{j-1/2})(1-e^{-2\pi
\sqrt{-1}v}q^{j-1/2})], \ee
\be \theta_3(v,\tau)=\prod_{j=1}^\infty[(1-q^j)(1+e^{2\pi
\sqrt{-1}v}q^{j-1/2})(1+e^{-2\pi \sqrt{-1}v}q^{j-1/2})], \ee where
$q=e^{2\pi \sqrt{-1}\tau}, \tau\in \mathbb{H}$.

They are all holomorphic functions for $(v,\tau)\in \mathbb{C \times
H}$, where $\mathbb{C}$ is the complex plane and $\mathbb{H}$ is the
upper half plane.
The theta functions satisfy the the following
transformation laws (cf. \cite{Ch85}), 
\be \label{theta-tran1}
\theta(v,\tau+1)=e^{\pi \sqrt{-1}\over 4}\theta(v,\tau),\ \ \
\theta\left(v,-{1}/{\tau}\right)={1\over\sqrt{-1}}\left({\tau\over
\sqrt{-1}}\right)^{1/2} e^{\pi\sqrt{-1}\tau v^2}\theta\left(\tau
v,\tau\right)\ ;\ee 
\be \theta_1(v,\tau+1)=e^{\pi \sqrt{-1}\over
4}\theta_1(v,\tau),\ \ \
\theta_1\left(v,-{1}/{\tau}\right)=\left({\tau\over
\sqrt{-1}}\right)^{1/2} e^{\pi\sqrt{-1}\tau v^2}\theta_2(\tau
v,\tau)\ ;\ee 
\be\theta_2(v,\tau+1)=\theta_3(v,\tau),\ \ \
\theta_2\left(v,-{1}/{\tau}\right)=\left({\tau\over
\sqrt{-1}}\right)^{1/2} e^{\pi\sqrt{-1}\tau v^2}\theta_1(\tau
v,\tau)\ ;\ee 
\be\theta_3(v,\tau+1)=\theta_2(v,\tau),\ \ \
\theta_3\left(v,-{1}/{\tau}\right)=\left({\tau\over
\sqrt{-1}}\right)^{1/2} e^{\pi\sqrt{-1}\tau v^2}\theta_3(\tau
v,\tau)\ .\ee

\begin{equation} \label{theta-tran2}
\theta (v+1,\tau )=-\theta (v,\tau ),\ \theta (v+\tau ,\tau )=-%
e^{-\pi \sqrt{-1}(\tau+2v)}\theta (v, \tau ),
\end{equation}%
\begin{equation}
\theta _{1}(v+1,\tau )=-\theta _{1}(v,\tau ),\ \theta _{1}(v+\tau ,\tau )=%
e^{-\pi \sqrt{-1}(\tau+2v)}\theta _{1}(v,\tau ),
\end{equation}%
\begin{equation}
\theta _{2}(v+1,\tau )=\theta _{2}(v,\tau ),\ \theta _{2}(v+\tau ,\tau )=-%
e^{-\pi \sqrt{-1}(\tau+2v)}\theta _{2}(v,\tau ),
\end{equation}~
\begin{equation}
\theta _{3}(v+1,\tau )=\theta _{3}(v,\tau ),\ \theta _{3}(v+\tau ,\tau )=%
e^{-\pi \sqrt{-1}(\tau+2v)}\theta _{3}(v,\tau ).
\end{equation}

\end{document}